\documentclass[12pt]{article} 
\usepackage{latexsym,amsfonts,amssymb,amsmath} 
\newtheorem{thm}{Theorem}[section] 
 
\newtheorem{prp}[thm]{Proposition} 
\newtheorem{lm}[thm]{Lemma} 
\newtheorem{rmk}[thm]{Remark} 
\newtheorem{ex}[thm]{Example} 
\newtheorem{dfn}[thm]{Definition} 
\newenvironment{proof}{\noindent{\bf Proof.}}{\noindent$\Box$\par\medskip}

\newcommand{\A}{{\mathcal A}}
\newcommand{\B}{{\mathcal B}}
\newcommand{\M}{{\mathcal M}}
\newcommand{\lie}{{\mathcal L}}
\newcommand{\F}{{\mathcal F}} 

\newcommand{\supp}{\mathrm {supp}}

\newcommand{\im}{\mathrm {im}}

\newcommand{\sudda}[1]{}

\begin{document}
\date{} 
\title{Decomposition theorems for a generalization of the holonomy Lie algebra of an arrangement}   
\author {Clas L\"ofwall}
\maketitle
\begin{abstract} 
In the paper "When does
the  associated graded Lie algebra of an arrangement group
decompose?" by Stefan Papadima and Alexander Suciu, it is proved that the holonomy Lie algebra of an
arrangement of hyperplanes through origin decomposes as a direct product of
Lie algebras in degree at least two if and only if a certain
(computable) condition
is fulfilled. We prove similar results for a
class of Lie algebras which is a generalization of the holonomy Lie
algebras. The proof methods are the same as in the paper cited above. 
\end{abstract}
\section{Introduction} The holonomy Lie algebra of a central
hyperplane arrangement in $\mathbb C^l$ is defined by generators and relations as
follows (see \cite{den}). Let $x_1,\ldots,x_n$ be the names of $n$
hyperplanes which serve as generators of a free
(ordinary, i.e., non-super) Lie algebra over  the
complex numbers. For all maximal subsets
$\{y_1,\ldots,y_k\}$ of the  hyperplanes
whose intersection has codimension two, there are 
relations $[y_i,\sum_{j=1}^ky_j]$ for $i=1,\ldots,k$. The Koszul dual
(see \cite{fro}) of
the enveloping algebra of the holonomy Lie algebra is the exterior algebra on $e_1,\ldots, e_n$ dual to
the generators $x_1,\ldots,x_n$, modulo the following relations. For
each triple $a,b,c$ of the generators, corresponding to a triple of
hyperplanes whose intersection has codimension two, there is the
relation $ab-ac+bc$. The Orlik-Solomon algebra $\mathcal O$, i.e., the cohomology
of the complement of the union of the hyperplanes, is a quotient of
this algebra with relations of higher degrees obtained in a similar
way (see \cite{den}). The enveloping algebra of the holonomy Lie algebra is the subalgebra
generated by the elements of degree one in the  Yoneda algebra
$\mathrm{Ext}_{\mathcal O}(\mathbb
C,\mathbb C)$, (see \cite {lof}).  The holonomy Lie algebra may be
defined over any coefficient ring. There is a result by Kohno
(\cite{koh}) that the holonomy Lie algebra over the rationals is equal
to the graded associated Lie algebra of the fundamental group of the
arrangement tensored with the rationals.

The holonomy Lie algebra is composed of ``localized" Lie algebras
obtained by restricting to the hyperplanes in a maximal set of
hyperplanes whose intersection has codimension two. These Lie algebras
are free, in degree at least two, on $k-1$ generators, where $k$ is the size of
the maximal set.  
In \cite{pap} it is studied under what circumstances the holonomy Lie algebra 
decomposes
as a direct product of these localized Lie algebras in degree at least two. 

We will generalize this by studying any set of ``localized" Lie
algebras (not even graded) with the only and important property that two different sets of
generators
have at most one generator in common, which obviously is the case for
holonomy Lie algebras. We will prove a decomposition theorem
for ``closed subarrangements" introduced in \cite{pap} and we also get
decomposition results in the case when the Lie algebra is not fully
decomposable. Finally, we give an application with an example of the Lie structure of a non-decomposable
holonomy Lie algebra.

\vspace{10pt}
\noindent{\bf Acknowledgement} This work originates from an intense
cooperative work during the fall 2013 with Jan-Erik Roos and Samuel
Lundqvist, when we tried to find the Lie structure of different
non-decomposable holonomy Lie algebras.

\section{General theory} 
Let $C$ be a commutative ring with unit.
\begin{dfn}\label{arr} A {\rm set-arrangement} on
$X=\{x_1,\ldots,x_n\}$, where $n\geq2$, is a subset $\A$ of all subsets of
$X$ such that each set in $\A$ has at least
two elements,  
two different sets in $\A$ has  at most one element in common, and the
union of all subsets in ${\A}$ is equal to $X$. 
\end{dfn}
A central hyperplane arrangement gives rise to a set-arrangement (see definition below). Consider the set $X$ of all normals of the hyperplanes in a central arrangement. The set of subsets of $X$ consisting of independent normals form a simple matroid, i.e., a matroid which contains all two-elements sets. Any simple matroid defines a set-arrangement consisting of all 2-flats, i.e., the maximal sets which contain no three-elements sets from the matroid. This is indeed a set-arrangement, since the intersection of two different 2-flats is again a flat and since it is not a 2-flat it consists of at most one element. (This is a part of the geometrical lattice consisting of all flats, see \cite{stan}). Also the opposite holds which we formulate as a proposition.

\begin{prp}\label{mat} Let $\A$ be a set-arrangement on $X$, containing every two-elements subset of $X$ which is not a proper subset of any set of $\A$. Then $\A$ is the set-arrangement of all 2-flats of the simple matroid $\M$ consisting of all subsets of $X$ with at most two elements together with all three-elements subsets of $X$ that are not contained in any set of $\A$.
\end{prp}
\begin{proof} We first prove that $\M$ is a matroid. To simplify reading we introduce the notation $xyz$ $independent$ for $\{x,y,z\}\in\M$ and $xyz$ $dependent$ for $\{x,y,z\}\notin\M$. Suppose $|\{x_1,x_2,x_3\}|=3$, $x_1x_2x_3$ independent, $a,b\in X$ and suppose $abx_i$  dependent for $i=1,2,3$. Then for each $i$ there is $A_i\in\A$ such that $\{a,b,x_i\}\subseteq A_i$. But $a\ne b\in A_i$ for all $i$ implies $A_1=A_2=A_3$. Hence  $\{a,b,x_1,x_2,x_3\}\subseteq A_1$ which gives $x_1x_2x_3$ dependent which is a contradiction. (The case when $a=x_1$ is treated similarily.) Hence $\M$ is a simple matroid.

Next we prove that if $A$ is a 2-flat of $\M$ then $A\in\A$. Suppose $A=\{x,y\}$ is a 2-flat and that $A\subset B$ and $B\in\A$. Then for some $z\in X$, $z\ne x,y$, $\{x,y,z\}\subseteq B$ and hence $xyz$ is dependent which contradicts the fact that $A$ is a 2-flat and hence by the assumption on $\A$ we get $A\in\A$. Suppose $A$ is a 2-flat with at least three elements $x,y,z$. Then $xyz$ is dependent and hence there is $B\in\A$ such that $\{x,y,z\}\subseteq B$. Suppose $u\in A$ and $u\ne x,y$. Then $xyu$ is dependent and hence there is $C\in\A$ such that $\{x,y,u\}\subseteq C$. Since $|B\cap C|\ge2$ we must have $C=B$. It follows that $A\subseteq B$. Suppose $v\in B$, $v\ne x,y$. Then $xyv$ is dependent and hence $v\in A$, since $A$ is a 2-flat, which implies that $A=\{x,y\}\cup\{w;\ xyw \text{ dependent}\}$.

Now suppose $A\in\A$. We want to prove that $A$ is a 2-flat of $\M$. If $|A|=2$, then $A\not\subset B$ where $|B|=3$ and $B$ dependent, and hence $A$ is a 2-flat. Suppose $|A|\ge3$. Let $x\ne y\in A$ and let $v\in X$, $v\ne x,y$. We want to prove that $xyv\text{ dependent} \iff v\in A$. Suppose $xyv\text{ dependent}$. Then $\{x,y,v\}\subseteq B$ for some $B\in\A$ and we must have $B=A$ and thus $v\in A$. Suppose $v\in A$. Then $\{x,y,v\}\subseteq A$ and hence $xyv\text{ dependent}$.

\end{proof}

\begin{dfn}\label{lie} A Lie algebra $\lie$ 
of a set-arrangement $\A$ on \hfill\break $X=\{x_1,\ldots,x_n\}$ is a free
ordinary (i.e., non-super)
Lie algebra, $\mathcal{F}(X)$, with coefficients in $C$, modulo
all relations $[x_i,x_j]$ where $\{x_i,x_j\}$ is not a subset
of any $A\in \A$ together with a set of  
relations $R$, $R=\cup_{A\in \A}
R_A$, where $R_A\in \mathcal{F}(A)'=[\mathcal{F}(A),\mathcal{F}(A)]$. The Lie algebra $ \mathcal{F}(A)/\langle R_A\rangle$ is called the localized
Lie algebra of $\lie$ at $A$ and denoted $\lie_A$.
\end{dfn}
 \begin{dfn}\label{set} The corresponding set-arrangement  of a hyperplane arrangement $x_1,\ldots,x_n$ is defined as the set of all maximal subsets $\{y_1,\ldots,y_k\}$ of $\{x_1,\ldots,x_n\}$, whose intersection has codimension two.  
\end{dfn}
The holonomy Lie algebra of a hyperplane arrangement (defined in the introduction) is a Lie algebra of the corresponding set-arrangement, where $[x_i,x_j]=0$ for any two-set $\{x_i,x_j\}$  in the set-arrangement.

\begin{rmk} {\rm If a set-arrangement $\A$ on $X$ is enlarged to $\hat\A$ by adding two-elements subsets of $X$, which are not contained in any set of $\A$, then a Lie algebra of $\A$ is the same as a Lie algebra of $\hat\A$, where the variables in the new two-elements sets commute. On the other hand if $\{x,y\}\in\A$ and $\lie$ is a Lie algebra of $\A$ such that $[x,y]=0$ then $\lie$ is also a Lie algebra of the arrangement $\A'=\A\setminus\{x,y\}$ (if $\A'$ still is a set-arrangement on $X$). For this reason, we may assume that a Lie algebra of a set-arrangement containing a two-set $\{x,y\}$ satisfies $[x,y]\ne0$.} 
\end{rmk}

\vspace{12pt}
Let  $\lie$ be a Lie algebra of a set-arrangement
$\A$. There is a natural Lie algebra morphism
$s_A:\lie_A\to \lie$ for each $A\in\A$. Also, we
may define a Lie algebra morphism $\pi_A:\lie\to\lie_A$
by sending all variables not in $A$ to zero (and a variable in $A$
to itself). This is well-defined, since if $r\in R_B$ where $B\ne A$,
then each Lie monomial in $r$ contains at least one variable which
is not in $A$, since $|A\cap B|\leq1$ and $[x,x]=0$ for all $x\in\lie$; also if for all
$B\in\A$,
$\{a,b\}\not\subseteq B$, then in particular $a\notin A$ or $b\notin A$ and hence
$\pi_A([a,b])=0$.

We have $\pi_A\circ s_A={\rm id}_{\lie_A}$ for each
$A\in\A$. Now consider the derived functor
$\lie'=[\lie,\lie]$. We have for each
$A\in\A$, $s_A':\lie_A'\to
\lie'$, $\pi_A':\lie'\to\lie_A'$ and $\pi_A'\circ
s_A'={\rm id}_{\lie_A'}$. Define $C$-module morphisms $s'$ and
$\pi'$ as
\begin{align*} 
s' &=\sum_{A\in\A}s_A':\ \bigoplus_{A\in\A}\lie_A'\to \lie'\\
\pi'&= \bigoplus_{{A\in\A}}\pi_A':\ \lie'\to\bigoplus_{A\in\A}\lie_A'
\end{align*}
\begin{prp}\label{inj} We have $\pi_B'\circ s_A'=0$ if $B\ne A$,
$\pi'\circ s'={\rm id}_{\oplus\lie_A'}$,
$\pi'$ is surjective, $s'$ is injective, and the $C$-submodule, $\im(s')=\sum\im(s_A')$,  of
$\lie'$ is a direct sum.
\end{prp}
\begin{proof} It is enough to prove the first assertion. This follows
 in the same way as above. Indeed, an element in $\im(s_A')$ is represented by an element in $\F(A)'$, in
which
each Lie monomial contains at least one generator which is not in $B$.
\end{proof}
\begin{thm}\label{eq}The following are equivalent for a Lie algebra $\lie$ of a
set-arrangement $\A$ on $X$.
\begin{align*}
1)&\quad[x,\im(s_A')]=0\quad\text{for all }A\in\A\text{
and all }x\in X\setminus A\\
2)&\quad \im(s_A')\quad\text{is an ideal in }\lie\text{ for
all }A\in\A\\
3)&\quad\lie'=\sum_{a\in\A}\im(s_A')\\
4)&\quad s'\text{ is surjective}\\
5)&\quad \pi'\text{ is injective} 
\end{align*}
If the conditions above are satisfied, then
$$
\lie'= \bigoplus_{A\in\A}\im(s_A')\cong\bigoplus_{A\in\A}\lie_A'
$$
as Lie algebras.
\end{thm}
\begin{proof}
1) $\Rightarrow$ 2): obvious, since $[x,\im(s_A')]\subseteq{\rm
im}(s_A')$ if $x\in A$.

2) $\Rightarrow$ 3): By the properties of the defining ideal for
$\lie$, it follows that all elements in $\lie'$ which are
represented by $[x,y]$, where $x,y\in X$, belong to $\sum{\rm
im}(s_A')$. But as an ideal in $\lie$, $\lie'$ is
generated by such elements and hence 3) follows from 2).

3) $\Rightarrow$ 4): obvious

4) $\Rightarrow$ 5): $s'$ is an isomorphism by 4) and Proposition
\ref{inj}. Since $\pi'\circ s'={\rm id}_{\oplus\lie_A'}$ it
follows that $\pi'$ is the inverse of $s'$.

 5) $\Rightarrow$ 1): We have $\pi_A'([x,\im(s_A')])=0$ if $x\notin
A$. Also, by Proposition \ref{inj}, $\pi_B'([x,{\rm
im}(s_A')])=0$ if $B\ne A$. Hence $\pi'([x,\im(s_A')])=0$ if
$x\not\in A$ and hence 1) follows from 5).

The last statement follows by Proposition \ref{inj} and since by 1) and 2) $[\im(s_A'),{\rm
im}(s_B')]=0$ if $A\ne B$. 
\end{proof}

In practise, one might want to check the condition 1) by a computation. If however $\im(s_A')$ is infinite dimensional, this will not give a proof. We will now introduce a condition on $\lie$, which is fullfilled for holonomy Lie algebras, which makes it possible to replace 1) in Theorem  \ref{eq} by the computational effective condition 6) in Theorem \ref{comp} below.

\begin{dfn}\label{rep} A Lie algebra $\lie$ of a set-arrangement $\A$ is said to satisfy the {\rm replacement} condition if for all
$A\in\A$ with $|A|\ge3$ and all $x\in A$ such that for some $B\in\A$, $x\in B$ and $B\ne A$, it holds that $\lie_A'$ is contained in
the Lie subalgebra of $\lie_A$ generated by
$A\setminus \{x\}$.
\end{dfn}

A holonomy Lie algebra satisfies the replacement condition since if
$A=\{y_1,\ldots,y_k\}$ and $x=y_j$ then for all $i\ne j$ we have the relation
$[x,y_i]=\sum_{r\ne j}[y_i,y_r]$ and hence for all $x\in A$, $\lie_A'$ is contained in
the Lie subalgebra of $\lie_A$ generated by
$A\setminus \{x\}$.

\begin{thm}\label{comp} Let $\lie$ be a Lie algebra of a
set-arrangement $\A$ on $X$ and suppose $\lie$ satisfies the
replacement condition and that all two-elements sets in $\A$ are disjoint. Suppose also either that

\vspace{10pt}
\noindent\rm{(i)}\ $\lie$ is graded by giving the generators in $X$ some positive integer degrees

\vspace{6pt}
\noindent or

\vspace{6pt}
\noindent(ii) \ For all $A\in\A$ there is at most one $B\in\A$ such that $|B|=2$ and $A\cap B\ne\emptyset$.
 
\vspace{10pt}
\noindent Then the conditions 1) to 5) in Theorem
\ref{eq} are equivalent to
\begin{align*}
6)\quad[x,[y,z]]=0\text{ for all }A\in\A\text{ and all
}y,z\in A \text{ and }
x\in X\setminus A
\end{align*} 
\end{thm}
\begin{proof} Of course, we only have to prove 6) implies 1). Suppose
6) is true. Suppose (i) holds. We prove by induction over the degree of $r$ that $[x,r]=0$  if $x\in X\setminus A$ and $r\in\im(s_A')$. This is true if the degree of $r$ is two by 6). For each $x$ there is at most one $y_x\in A$ such that $\{x,y_x\}\in\A$, and if $y_x$ exists, we must have $|A|\ge3$. Hence, given $x$, we can use replacement to express $r\in\im(s_A')$ in terms of the generators in $A\setminus\{y_x\}$. Suppose now $[x,r]=0$ for all $r\in\im(s_A')$ of degree $\le n$ and all $x\in X\setminus A$. Let $x\in X\setminus A$ be given. Express an element in $\im(s_A')$ of degree $n+1$ in terms of the generators in $A\setminus\{y_x\}$. This element is a linear combinations of elements of the form $[a,r]$ where $a\in A\setminus\{y_x\}$ and $r\in\im(s_A')$ is of degree $\le n$. We want to prove that $[x,[a,r]]=0$. It is enough to prove that $[[x,a],r]=0$. If $[x,a]\ne0$ there is a $B\in\A$ such that $x,a\in B$. Since $a\ne y_x$, we have $|B|\ge3$ and we may use replacement to write $[x,a]=s$ where $s$ is expressed in terms of the generators in $B\setminus \{a\}$. But $A\cap B=\{a\}$. Hence we may use Jacobi and induction to deduce that $[s,r]=0$.

Suppose (ii) holds. We now use induction over word length with respect to a fix representation of elements in $\im(s_A')$ in terms of a choice of generators as follows. By (ii) there is at most one $y\in A$ such that $\{x,y\}\in \A$ for some $x\in X\setminus A$. If $y$ exists we must have $|A|\ge3$. We choose the generators in $A\setminus\{y\}$ (or just $A$ if $y$ does not exist) to express  elements in $\im(s_A')$. Suppose now $[x,r]=0$ for all $x\in X\setminus A$ and all monomials $r$ of word length $\le n$ in the generators in $A\setminus\{y\}$. It is enough to prove that $[[x,a],r]=0$ if $a\in A\setminus\{y\}$ and $r\in\im(s_A')$ is of word length $\le n$. If $[x,a]\ne0$ there is a $B\in\A$ such that $x,a\in B$. Since $a\in A$ and $a\ne y$, we have $|B|\ge3$ and we may finish the proof as above.
\end{proof}

\sudda{
We prove by induction over the word length of $r$ that
$[x,r]=0$ if $x\in X\setminus A$ and $r$ is a Lie monomial, of word length at
least two, built from
generators in $A$. This is true when the word length of $r$ is two by
6). Suppose now $r$ is a Lie monomial such that $[x,r]=0$ for all
$x\in X\setminus A$. Suppose $a\in A$. We want to  prove that $[x,[a,r]]=0$
for all $x\in X\setminus A$. By Jacobi and induction, it is enough to prove
that $[[x,a],r]=0$. If $[x,a]\ne0$, there is a $B\in\A$, $B\ne
A$ such that $x,a\in B$. By replacement, there is an $s$ belonging to
the Lie subalgebra of $\lie_B$ generated by
$B\setminus \{a\}$ such that $[x,a]=s$. Since $B\cap A=\{a\}$, the variables occuring in
$s$ do not belong to $A$. Hence by Jacobi and induction $[s,r]=0$.  

 If the degree of $r$ is 1, there is nothing to prove. Suppose this is true if the degree of $r$ is $<n$ and suppose that the degree of $r$ is $n$. Let $x\in X\setminus A$. Suppose $\{x,y\}\in\A$ where $y\in A$ (if this is not the case we may continue as in (i)). Then by assumption $|A|\ge3$ and we may use replacement to write $r$ as a linear combination of Lie monomials not containing $y$.}

\begin{ex}\label{ex1}  
\begin{align*}X&=\{x,y,z,u,v\},\quad A=\{x,y\},\quad B=\{x,z,u\},\quad
C=\{y,z,v\},\\
\A&=\{A,B,C\},\quad
R_A=\{\},\quad
R_B=\{[z,x]-[u,[u,x]],[z,u]-[u,x]\},\\
R_C&=\{[z,y]-[v,[v,y]],[z,v]-[v,y]\}.
\end{align*}

\rm{
The Lie algebra defined by this example has also the relations $[x,v]=[y,u]=[u,v]=0$ and we claim that it satisfies the replacement condition. Indeed, the form of the relations $R_B$ shows that $\lie_B'$ is contained in the Lie subalgebra of $\lie_B$ generated by $u,x$. Moreover, by writing $R_B$ as $\{[z,x]-[u,[z,u]],[u,x]-[z,u]\}$ we see that $\lie_B'$ is contained in the Lie subalgebra of $\lie_B$ generated by $u,z$. The lie algebra $\lie_C$ is treated similarily. Moreover, the Lie algebra is graded by giving the generators the degrees 2,2,2,1,1, so both (i) and (ii) in Theorem \ref{comp} are satisfied.
To be able to apply Theorem \ref{comp} we only have to check that $[[x,y],z]=[[x,y],u]=[[x,y],v]=
[[x,z],y]=[[x,u],y]=[[x,z],v]=[[x,u],v]=[[y,z],x]=[[y,v],x]=
[[y,z],u]=[[y,v],u]=0$. This may be done by hand (with some effort) or by using the program \texttt{liedim} \cite{lie}.
Hence we get that the derived Lie algebra is isomorphic
to the direct product of the derived Lie algebras
$\lie_A'$,$\lie_B'$,$\lie_C'$. It is easily seen that they all are free on two generators (in degree $\ge2$).}
\end{ex}
  
\section{Closed sub-arrangements}
\begin{dfn}If $\A$ is a set-arrangement and $\B\subseteq\A$, $\supp(\B)$ is defined as
$\cup_{B\in\B}B$. Then $\B$ is a set-arrangement
on  $\supp(\B)$. 
\end{dfn}
Following \cite{pap} we make the definition of a ``closed"
sub-arrangement.
\begin{dfn}If $\A$ is a set-arrangement and $\B\subseteq \A$, $\B$ is called a {\rm closed} sub-arrangement of
$\A$ if for all $A\in\A\setminus\B$,
$|A\cap\supp(\B)|\le1$. 
\end{dfn}
Given a closed sub-arrangement $\B$ of $\A$ and a Lie
algebra $\lie$ of $\A$, we define the localized Lie
algebra at $\B$ as $\lie_{\B}=\F(\supp(\B))/\langle R_{\B}\rangle$, where 
$$
R_{\B}=\cup_{B\in\B}R_B\cup\{[x,y];\ x,y\in\supp(\B) \text{ and }\forall B\in\B\ \{x,y\}\not\subseteq B\}
$$
This extends the old definition of a localized Lie algebra since
$\{A\}$ is a closed sub-arrangement of $\A$ for all
$A\in\A$ and $R_{\{A\}}=R_A$.

By the ``closed" condition, there is a natural Lie algebra map
$$
s_{\B}:\ \lie_{\B}\to \lie
$$
A map $\F(\supp(\A))\to  \lie_{\B}$
is defined by sending the generators $x\notin\supp(\B)$
to zero (and not changing the other generators). We want to prove that
the map factors through $\lie$. If for all $A\in\A$,
$\{x,y\}\not\subseteq A$, then $[x,y]=0$ in $\lie_{\B}$. Moreover, suppose $r\in R_{A}$, $A\notin\B$, then any
non-zero Lie monomial in $r$ must contain at least one generator not in $\supp(\B)$, since $\B$ is closed. Hence $r$ is
mapped to zero and we have defined a map
$$
\pi_{\B}:\ \lie\to \lie_\B
$$
We have $\pi_{\B}\circ s_{\B}={\rm id}_{\lie_\B}$. 
Hence, $s_\B$ is injective and $\pi_\B$ is surjective
and $\lie$ is a semidirect product of $\im(s_\B)$ operating onto $\ker(\pi_\B)$
and $\ker(\pi_\B)$ is the ideal generated by $\supp(\A)\setminus{\supp}(\B)$.

\begin{prp}\label{alg} Let $\B$ be a closed sub-arrangement of $\A$ and let $\lie$ be a Lie algebra of $\A$. Suppose $\lie$ satisfies replacement and there is no $A\in\A$, $|A|=2$ and $|A\cap\supp(\B)|=1$. Then $\ker(\pi_\B)$ is the Lie subalgebra of $\lie$ generated by $\supp(\A)\setminus{\supp}(\B)$.
\end{prp}
\begin{proof} Let $S$ be the Lie subalgebra of $\lie$ generated by $\supp(\A)\setminus{\supp}(\B)$. Since $S\subseteq\ker(\pi_\B)$ and $\ker(\pi_\B)$ is generated by $\supp(\A)\setminus{\supp}(\B)$, it is enough to prove that $S$ is an ideal. We prove, by induction over the word length $n$ of a Lie monomial $r$ in $S$, that $[b,r]\in S$ for all $b\in\supp(\B)$. 
Suppose $n=1$ and $[b,r]\ne0$. Then there is $A\in\A$ such that $b,r\in A$. Since $A\cap\supp(\B)=\{b\}$ we must have $|A|\ge3$ by assumption. Hence we may use replacement to get $[b,r]\in S$. Suppose the claim is proved for word length $\le n$ and let $r\in S$ be a Lie monomial of word length $n$. We must prove that $[b,[a,r]]\in S$ for all $a\in\supp(\A)\setminus\supp(\B)$ and all $b\in\supp(\B)$. By Jacobi and induction it is enough to prove that $[[b,a],r]\in S$. If $[b,a]\ne0$ there is $A\in\A$ such that $b,a\in A$ and we can continue as above to get $[b,a]\in S$.

\end{proof}
\begin{prp}\label{dir} Let $\B$ be a closed sub-arrangement of $\A$ and let $\lie$ be a Lie algebra of $\A$. Then
$\lie'$ is the direct product of $\im(s_\B')$ and
$\ker(\pi_\B')$ if
$$
[x,\im(s_\B')]=0\quad\text{for all }x\in\supp(\A)\setminus{\supp}(\B)
$$ 
\end{prp}
\begin{proof} Obvious, since by Jacobi and induction we get, $[\ker(\pi_\B'),\im(s_\B')]=0$.
\end{proof}

Suppose now we have pairwise disjoint closed sub-arrangements
$\B_i$, $i=1,\ldots,k$ such that $\cup\B_i=\A$. (This gives a partition of $\A$ and not of $\supp(\A)$.) As in the previous section  we may define $C$-module morphisms $s'$ and
$\pi'$ as
\begin{align*} 
s' &=\sum_{i=1}^k s_{\B_i}':\
\bigoplus_{i=1}^k\lie_{\B_i}'\to \lie'\\
\pi'&= \bigoplus_{i=1}^k\pi_{\B_i}':\ \lie'\to\bigoplus_{i=1}^k\lie_{\B_i}'
\end{align*}
\begin{prp}\label{subinj} We have $\pi_{\B_i}'\circ s_{\B_j}'=0$ for $i\ne j$, $\pi'\circ s'={\rm
id}_{\oplus\lie_{\B_i}'}$,
$\pi'$ is surjective, $s'$ is injective and the $C$-submodule, ${\im}(s')=\sum\im(s_{\B_i}')$,  of
$\lie'$ is a direct sum.
\end{prp}
\begin{proof} We only need to prove the first assertion. It may happen
that $\supp(\B_i)$ and $\supp(\B_j)$ have
several elements in common, but these elements commute, since if
 $x,y\in\supp(\B_i)\cap{\rm
supp}(\B_j)$ and $[x,y]\ne0$, then there is $A\in\A$
such that $x,y\in A$. By the closed condition, it follows that
$A\in\B_i\cap\B_j$ which is a contradiction if $i\ne j$. Now we
can argue in the same way as in the proof of \ref{inj}. Indeed, in any
non-zero monomial built up by elements in $\supp(\B_j)$
there is at least one element which does not belong to ${\rm
supp}(\B_i)$ if $i\ne j$.
\end{proof} 
\begin{thm}\label{sub} Let $\B_i$, $i=1,\ldots,k$ be pairwise disjoint closed sub-arrange\-ments
  of a set-arrangement $\A$ on $X$, such that $\cup\B_i=\mathcal
A$. The following are equivalent for a Lie algebra $\lie$ of $\A$.
\begin{align*}
1)&\quad[x,\im(s_{\B_i}')]=0\quad\text{for all }i=1,\ldots,k\text{
and all }x\in X\setminus \supp(\B_i)\\
2)&\quad \im(s_{\B_i}')\quad\text{is an ideal in }\lie\text{ for
all }i=1,\ldots,k\\
3)&\quad\lie'=\sum_{i=1}^k\im(s_{\B_i}')\\
4)&\quad s'\text{ is surjective}\\
5)&\quad \pi'\text{ is injective} 
\end{align*}
If the conditions above are satisfied, then
$$
\lie'= \bigoplus_{i=1}^k\im(s_{\mathcal
B_i}')\cong\bigoplus_{i=1}^k\lie_{\B_i}'
$$
as Lie algebras.
\end{thm}
\begin{proof} Very much the same as the proof of \ref{eq}. In the
proof of 5) $\Rightarrow$ 1) we use Proposition \ref{subinj}. 
\end{proof}
To be able to prove the analogue of Theorem \ref{comp} we need a Lemma.
\begin{lm}\label{lem} Let $\B$ be a closed sub-arrangement of $\A$ and let $\lie$ be a Lie algebra of $\A$. Suppose $\lie$ satisfies the replacement condition and let $x\in\supp(\B)$ but $x\notin B$ where $B\in\B$ and $|B|=2$. Then $\im(s_{\B}')$ is contained in the Lie subalgebra generated by $\supp(\B)\setminus\{x\}$.
\end{lm}
\begin{proof} This is easily proven by induction.
\end{proof} 

\begin{thm}\label{six} Let $\B_i$, $i=1,\ldots,k$ be pairwise disjoint closed sub-arrange\-ments
  of a set-arrangement $\A$ on $X$, such that $\cup\B_i=\mathcal
A$.
Suppose that $\lie$ satisfies the replacement condition, and that all two-elements sets in $\A$ are disjoint. Suppose also either that

\vspace{10pt}
\noindent{\rm(i)}\ $\lie$ is graded by giving the generators in $X$ some positive integer degrees

\vspace{6pt}
\noindent or

\vspace{6pt}
\noindent{\rm(ii)} \ For all $i=1,\ldots,k$ there is at most one $B\in\A$ such that $|B|=2$ and $\supp(\B_i)\cap B\ne\emptyset$.
 
\vspace{10pt}
\noindent Then the conditions 1) to 5) in Theorem 
\ref{sub} are equivalent to
\begin{align*}
6)\ [x,[y,z]]=0\text{ for all }
y,z\in \supp(\B_i),\
x\in X\setminus\supp(\B_i),\ i=1,\ldots,k 
\end{align*} 
\end{thm}
\begin{proof}
This is almost the same as the proof of Theorem \ref{comp}. One only
has to replace $A$ by $\supp(\B_i)$, $B\ne A$ by
$B\notin\B_i$, $\im(s_A')$ by $\im(s_{\B_i}')$  and use that $\B_i$ is closed. The beginning of the proof goes as follows.

Suppose 6) is true and that (i) holds. We prove by induction over the degree of $r$ that $[x,r]=0$  if $x\in X\setminus \supp(\B_i)$ and $r\in\im(s_{\B_i}')$. This is true if the degree of $r$ is two by 6). For each $x$ there is at most one $y_x\in \supp(\B_i)$ such that $\{x,y_x\}\in\A$, and we must have $y_x\notin A\in\B_i$ where $|A|=2$ and hence we may use Lemma \ref{lem}. Hence, given $x$, we can use replacement to express $r\in\im(s_{\B_i}')$ in terms of the generators in $\supp(\B_i)\setminus\{y_x\}$. Now we can continue with the induction hypothesis as in the proof of Theorem \ref{comp} with the replacements given above and also follow the rest of the proof (and use Lemma \ref{lem} again in the proof of (ii)). 
\end{proof}

We end this section with a computable sufficient condition which gives
partial
decomposition in the case the Lie algebra satisfies the replacement
condition.
\begin{prp}\label{last}Let $\B$ be a closed sub-arrangement of $\mathcal
A$ and let $\lie$ be a Lie algebra of $\A$ which
satisfies the conditions given in Theorem \ref{six}. Then
$\lie'$ as a Lie algebra is the direct product of $\im(s_\B')$ and
$\ker(\pi_\B')$ if
$$
[x,[y,z]]=0\quad\text{for all }x\in\supp(\A)\setminus{\supp}(\B)\text{ and all }y,z\in\supp(\B)
$$
\sudda{Moreover, $\ker(\pi_{\B})$ is the Lie {\rm subalgebra} 
generated by $\supp(\A)\setminus\supp(\B)$.}  
\end{prp}
\begin{proof}
The claim follows from Proposition \ref{dir} and the proof of
Theorem \ref{six} since the proof there is valid for
each $i$ separately. \sudda{The last claim is proved in the same way as
Theorem \ref{six}.} 
\end{proof} 
\section{An example}
As an application of Theorem \ref{six}, we will study the holonomy Lie algebra of the hyperplane arrangement
studied in \cite{stu} with notation $(10_3)_3$ (this example was shown
to me by Jan-Erik Roos). It is {\em not} decomposable in the sense of \cite{pap}. It consists of 10 planes
through origin
in $\mathbb C^3$. There are 10 combinations of three planes whose
intersection has
codimension 2 (and no larger sets). We let 6,7,8,1,9,4,5,2,3,10 be  the names
of the planes in the order given in \cite{stu}. We get the following
set-arrangement on $\{1,2,\ldots,10\}$.
\begin{align*}
\A&=\{\{1,2,3\},\{1,4,5\},\{2,4,6\},\{3,5,6\},\\
&\{1,7,8\},\{2,7,10\},\{3,7,9\},\{4,8,10\},\{5,9,10\},\{6,8,9\}\}
\end{align*}
Put $\B_1=\{\{1,2,3\},\{1,4,5\},\{2,4,6\},\{3,5,6\}\}$. It is
easy to see that $\B_1$ is a closed sub-arrangement of $\A$. Its Lie algebra is the Lie algebra of the graphic arrangement $K_4$, the
complete graph with four vertices, see \cite{pap}. Put $\B_i=\{A_i\}$, $i=2,\ldots,7$ where $A_2,\ldots,A_7$
are the remaining sets in $\A$. The holonomy Lie algebra,
$\lie$, has
for each set $\{a,b,c\}$ in $\A$ the relations
$[a,b]=[b,c]=[c,a]$ together with all relations $[x,y]$ such that
$\{x,y\}$ is not a subset of any set in $\A$. 

We  have that $\lie$ is graded by letting
all generators have degree one. In degree two we have the 10
generators $[1,2], [1,4],\ldots, [6,8]$. To be able to use Theorem
\ref{six}, we have to check that $[1,2],[1,4],[2,4],[3,5]$ are
annihilated by $7,8,9,10$ and that $[1,7]$ is annihilated by
$2,3,4,5,6,9,10$ and so on. In total there are 58 triple products that
must be proved to be zero. To do this one can use the program
\texttt{liedim} \cite{lie}; it can compute over any
prime field (but not over the integers). Over the rationals the
program shows that all the 58 products are zero and hence the Lie
algebra $\lie'$ is the direct sum of seven localized Lie
algebras. (In fact, a computation by hand proves that the desired products are zero also over the integers.)

To analyze
$\lie_{\B_1}$, put $A=\{1,2,3\}$. By Proposition
\ref{alg}, $\lie_{\B_1}'$ is a semidirect product of $\lie_A'$ operating onto
  the Lie subalgebra generated by $4,5,6$ (in degree $\ge2$). The Lie
algebra $\lie_A$ is also a semidirect product of the free Lie algebra
on $3$ operating onto the free Lie algebra on $1,2$, where $3.1=[1,2],\ 3.2=[2,1]$. Hence we get that
$\lie_{\B_1}$ is 
$\F(4,5,6)/I\rtimes(\F(1,2)\rtimes\F(3))$ for some ideal $I$.  On the other hand, by defining the
operations of 1,2,3 on $\F(4,5,6)$ and checking that the 
relations $[1,2]-[2,3], [1,2]-[3,1]$ operate as zero one can deduce that
also $\F(4,5,6)\rtimes(\F(1,2)\rtimes\F(3))$ is a quotient of $\lie_{\B_1}$ and hence $I=0$ and
we get 
$$
\lie_{\B_1}\cong \F(4,5,6)\rtimes(\F(1,2)\rtimes\F(3)) 
$$
Finally, we get the following decomposition of $\lie$.
$$
\lie\cong \bigl(\F(4,5,6)\rtimes(\F(1,2)\rtimes\F(3))\bigr)\oplus\bigoplus_{i=2}^7\lie_i
$$
where  $\lie_i\cong\F(1,2)\rtimes\F(3)$, $i=2,\ldots,7$. Also
$$
\lie'\cong (\F(4,5,6)'\rtimes\F(1,2)')\oplus\bigoplus_{i=2}^7\F(1,2)'
$$

\vspace{20pt}
\noindent
Clas L\"ofwall\\
Gim\aa{}gatan 5\\
12848 Bagarmossen\\
Sweden\\

\noindent
e-mail:\\
clas.lofwall@gmail.com
\end{document}